\documentclass[11pt]{amsart}

\usepackage[utf8]{inputenc}
\usepackage[english]{babel}
 
\usepackage{anysize}
\usepackage{amsfonts}
\usepackage{amssymb}
\usepackage{amsmath}
\usepackage{amsthm}
\usepackage[all]{xy}
\usepackage[alphabetic,backrefs,lite]{amsrefs} 
\usepackage{enumerate}
\usepackage{multirow}
\usepackage{geometry}
\usepackage{color}
\usepackage{colonequals}
\usepackage{mathrsfs}

\definecolor{darkgreen}{rgb}{0,0.5,0}

\definecolor{darkgreen}{rgb}{0,0.5,0}
\usepackage[
        draft=false,
        colorlinks, citecolor=darkgreen,
        backref,       
        pdfauthor={Filippo F. Favale and Sonia Brivio},
        pdftitle={Nodal curves and good polarizations},
        linktocpage]{hyperref}

\newtheorem{theorem}{Theorem}[section]
\newtheorem*{theorem*}{Theorem}
\newtheorem*{conjecture*}{Conjecture}
\newtheorem{proposition}[theorem]{Proposition}
\newtheorem{definition}[theorem]{Definition}
\newtheorem{notation}[theorem]{Notations}

\newtheorem{corollary}[theorem]{Corollary}
\newtheorem{lemma}[theorem]{Lemma}
\newtheorem{question}[theorem]{Question}
\newtheorem{remark}[theorem]{Remark}
\newtheorem{conjecture}[theorem]{Conjecture}
\newtheorem{example}[theorem]{Example}

\DeclareMathOperator{\Img}{Im}

\DeclareMathOperator{\Supp}{Supp}
\DeclareMathOperator{\Sing}{Sing}

\DeclareMathOperator{\Hom}{Hom}
\DeclareMathOperator{\Pic}{Pic}

\DeclareMathOperator{\supp}{Supp}
\DeclareMathOperator{\Tors}{Tors}

\newcommand{\OO}{\mathcal{O}}

\newcommand{\w}{{\underline w}}

\DeclareMathOperator{\wdeg}{deg_{\underline{w}}}
\DeclareMathOperator{\wrank}{rk_{\underline{w}}}

\newcommand{\cE}{\mathcal{E}}
\newcommand{\cV}{\mathcal{V}}
\newcommand{\cM}{\mathcal{M}}
\newcommand{\cP}{\mathcal{P}}

\numberwithin{equation}{section}

\begin{document}

\title[Nodal curves and good polarizations]{Nodal curves and polarizations with good properties}

\author{Sonia Brivio}
\address{Dipartimento di Matematica e Applicazioni,
	Universit\`a degli Studi di Milano-Bicocca,
	Via Roberto Cozzi, 55,
	I-20125 Milano, Italy}
\email{sonia.brivio@unimib.it}

\author{Filippo F. Favale}
\address{Dipartimento di Matematica e Applicazioni,
	Universit\`a degli Studi di Milano-Bicocca,
	Via Roberto Cozzi, 55,
	I-20125 Milano, Italy}
\email{filippo.favale@unimib.it}

\date{\today}
\thanks{
Both authors are partially supported by INdAM - GNSAGA.\\
}
 \maketitle
 
\begin{abstract}
In this paper we deal with polarizations on a nodal curve $C$ with smooth components. Our aim is to study and characterize a class of polarizations, which we call "good", for which depth one sheaves on $C$ reflect some properties that hold for vector bundles on smooth curves. We will concentrate, in particular, on the relation between the $\w$-stability of $\OO_C$ and the goodness of $\w$. We prove that these two concepts agree when $C$ is of compact type and we conjecture that the same should hold for all nodal curves.
\end{abstract}

\section*{Introduction}
Let $C$ be a projective curve over the complex field. One of the most interesting problems in Algebraic Geometry is the construction of moduli spaces parametrizing 
line bundles or in general vector bundles on $C$. These moduli spaces have been studied first by Mumford (\cite{M}) and Le Potier (\cite{LP}) in the smooth case. These spaces are interesting by themselves as higher dimensional varieties but also for important related constructions: just to mention some, one can consider higher-rank  Brill-Noether theory, Theta divisors and Theta functions and the moduli spaces of coherent systems. For surveys on these topics see, for example, \cite{Bra}, \cite{BGMN} and \cite{Bea}; for some results by the authors see \cite{BF1}, \cite{BF4}, \cite{Bri1}, \cite{Bri2}, \cite{BB12}  and \cite{BV}.
When the curve is singular, these spaces are not in general complete. It is natural to study their possible compactifications and this has driven the attention of many authors since the '60s, who addressed the problem with different approaches (see, for instance, \cite{Ses}, \cite{OS}, \cite{KN}, \cite{Bho92},\cite{Gie} and \cite{GM}). When $C$ is a reducible nodal curve, that is it has only ordinary double points, we have more explicit results. In several of the constructions mentioned above, the objects of these compact moduli spaces are equivalence classes of depth one sheaves (i.e. torsion free) on the curve that are semistable with respect to a polarization (see \cite{Tei91} and \cite{Tei95}). 
\\

A polarization $\w$ on $C$ is given by rational weights on each irreducible component of $C$ adding up to $1$ or, equivalently, by an ample line bundle $L$ on $C$ (see \cite{Ses} and \cite{KN}).
Once a polarization on the curve is fixed, the notions of degree and rank can be generalized to the notions of $\w$-degree and $\w$-rank which are also defined for depth one sheaves. With these data Seshadri introduced the notion of $\w$-stability (or $\w$-semistability) for depth one sheaves allowing the construction of moduli spaces of such objects.
\\

In this paper we are interested in studying polarizations on nodal reducible curves having nice properties, i.e. which allow us to generalize to nodal curves some natural properties of vector bundles on smooth curves and to simplify the study of stability of vector bundles and coherent systems on nodal reducible curves. As motivation, consider the following facts. On a smooth curve $C$, the sheaf $\OO_C$ is stable (as all line bundles) and any globally generated vector bundle has non-negative degree. This is not true anymore on reducible nodal curves. Moreover, in order to construct vector bundles on a reducible nodal curve, one can glue vector bundles on its components. In general, though, it is not true that glueing stable vector bundles yields a $\w$-stable sheaf: additional conditions on the polarization  and on the degree of the restrictions are needed (see \cite{Tei}, \cite{BF2} and \cite{BF3}).
\\

This motivates our definition of a good polarization. Let $C$ be a nodal curve with smooth irreducible components. For any depth one sheaf $E$ on $C$, we denote by $E_i$ the restriction (modulo torsion) of $E$ to the component $C_i$. Note that if $E$ is locally free, then the degree of $E$ is actually the sum of the degrees of its restrictions $E_i$, but 
this is not true in general. 
We will say that $\w$ is a {\it good polarization}  if  for any depth one sheaf $E$ the difference $\Delta_{\w}(E)$ of the $\w$-degree of $E$ and the sum of degrees of its restrictions $E_i$ is non negative and it is zero if and only if $E$ is locally free (see Definition \ref{GOOD}). As anticipated,  the first result of this paper is the following: 
\hfill\par
\begin{theorem*}[Theorem \ref{StableOC}]
\label{TH:A}
Let $C$ be a nodal curve and let $\w$ be a good polarization on it. Let $E$ be a depth one sheaf on $C$. Then we have the following properties:
\begin{enumerate}[(a)]
\item{} Assume that $E$ is locally free and, for $i=1,\dots,\gamma$, $E_i$ is stable with $\deg(E_i)= 0$.  Then $E$ is $\w$-stable.  
\item{} If $E$ is globally generated, then $\wdeg(E) \geq 0$. 
\item{} If $E$ is $\w$-semistable and $\wdeg(E)>0$,  then $h^0(E^*)=0$. 
\end{enumerate}
In particular, if $E=\OO_C$ or, more generally, if $E$ is a line bundle whose restrictions have degree $0$, then $E$ is $\w$-stable.
\end{theorem*}

We will show that good polarizations exist on any stable nodal curve with $p_a(C) \geq 2$ (see Proposition \ref{polarizationeta} and Corollary \ref{COR:familyw}). For nodal curves with $p_a(C)\leq 1$ we are able to characterize exactly which curves admit a good polarization (see Corollary \ref{COR:pa01}).
\\

The second result of this paper provides sufficient conditions in order to obtain a good polarization on a nodal curve. The method relies on the choice of particular paths on the dual graph $\Gamma_C$ of $C$ which yields a finite collection of subcurves $A_j$ of $C$. This allows us to get a rather technical description of $\Delta_{\w}(E)$, for any depth one sheaf $E$ on $C$,  and  to obtain the mentioned sufficient conditions. These are stated in Theorem \ref{THM:DELTAGENERAL}. More precisely, consider, for each non-empty subcurve $A_j$, the condition
\begin{equation*}
(\star\star)_{A_j}:\qquad \qquad \frac{1}{2}(\delta_{A_j}-1) < \Delta_{\w}(\OO_{A_j}) < \frac{1}{2}(\delta_{A_j}+1)
\end{equation*}
where $\delta_{A_j}$ is the number of the nodes of $C$ lying on $A_j$ which are not nodes for the subcurve $A_j$ (see Section \ref{SEC:NOTATIONS} for details). Then we have the following: 
\begin{theorem*}[Theorem \ref{THM:DELTAGENERAL}]
Let $(C,\w)$ be a polarized nodal curve. If conditions $(\star\star)_{A_j}$ hold for all non empty $A_j$, then  $\w$ is a good polarization. 
\end{theorem*} 

Motivated by many examples (some of them have been reported in Section \ref{SEC:EXAMPLE}),  we make this conjecture:
\begin{conjecture*}[Conjecture \ref{CONJ:OCwSTAB}]
Let $(C,\w)$ be a polarized nodal curve. Then  
$\OO_C$ is $\w$-stable if and only if $\w$ is a good polarization. 
\end{conjecture*}

In the third result of this paper we prove that this conjecture holds for curves of compact type:
\begin{theorem*}[Theorem \ref{THM:COMPACTTYPE}]
Let $(C,\w)$ be a polarized nodal curve of compact type. Then $\OO_C$ is $\w$-stable if and only if $\w$ is a good polarization.
\end{theorem*}

The idea is to prove that conditions $(\star\star)_{A_j}$ are  always implied  by stability of $\OO_C$ in the case of curves of compact type.
\\

Finally, we wonder how being a good polarization reflects on the line bundle inducing the polarization. This turns out to be related to the notion of balanced line bundles, as defined in \cite{Cap1}. Balanced line bundles are important tools when one has to deal with reducible nodal curves. For example, for such line bundles, a generalization of Clifford's Theorem holds. Our results can be summarized as (see Corollary \ref{COR:balanced} and Corollary \ref{COR:equivbalanced}):

\begin{theorem*}
Let $C$ be a stable nodal curve with $p_a(C) \geq 2$. Let $L$ be a line bundle of degree $p_a(C) -1$ and $\w$ be the polarization induced by $L$.
Then: 
\begin{enumerate}
    \item{} $L$ is  strictly balanced  if and only if 
    $\OO_C$ is $\w$-stable;
    \item{} if $C$ is of compact type, then $L$ is strictly balanced if and only if $\w$ is good.
    \end{enumerate}
\end{theorem*}
\vspace{2mm}

{\bf Ackowledgements}. The authors want to express their gratitude to both the anonymous referees for their helpful remarks and keen suggestions. They contributed a lot to the final version of this paper.

\section{Notations and preliminary results on nodal curves}
\label{SEC:NOTATIONS}

In this section we will introduce  notations and  we recall useful facts about nodal curves, their subcurves and polarizations.
\\

Let $C$ be a connected reduced nodal curve over the complex field (i.e. having only ordinary double points as singularities).  We will denote by $\gamma$ the number of  irreducible components and by $\delta$ the number of nodes  of $C$. We will assume that each irreducible component $C_i$ is a smooth curve of genus $g_i$. For the  theory of nodal curves see \cite[Ch X]{ACG}. 
We will denote by 
$$ \nu \colon  C^{\nu}= {\bigsqcup}_{i=1}^{\gamma}C_i \to C$$
the normalization map. If $p \in C$ is a node, we will denote by $q_{p,i_1}$  and $q_{p,i_2}$ the branch  points  over the node $p$, with $q_{p,i_k} \in C_{i_k}$. 
From the exact sequence:
$$ 0  \to \OO_C \to \nu_*\nu^*(\OO_C) \to \bigoplus_{p\in Sing(C)}{\mathbb C_p} \to 0, $$
we deduce that 
$\chi(\OO_C) = \sum_{i=1}^{\gamma}\chi(\OO_{C_i}) -\delta$,
and we obtain the {\it arithmetic  genus}  of $C$: 
\begin{equation}
\label{EQ:paC}
p_a(C)=\sum_{i=1}^{\gamma}g_i+\delta-\gamma+1.
\end{equation}
The {\it dual graph} of $C$  is the graph $\Gamma_C$ whose vertices are identified with the irreducible components of $C$  and whose edges are identified with the nodes of $C$.  An edge joins two  vertices if the corresponding node is in the intersection of the corresponding irreducible components.
So, $\Gamma_C$ has $\delta$ edges and $\gamma$ vertices, moreover it is connected since $C$ is connected.   Its first Betti number is
$b_1(\Gamma_C)=\delta -  \gamma + 1$. We recall that a connected nodal curve is said to be of {\it compact type} if every irreducible component of $C$ is smooth and its dual graph  is a tree. For  a curve of compact type  we have $\delta - \gamma +1 = 0$ and the pull-back   $\nu^*$ of the normalization map induces an isomorphism  
$\Pic (C) \simeq {\bigoplus}_{i=1}^{\gamma} \Pic(C_i)$ between the Picard groups.
\\ 

Let  $B$ be  a   proper subcurve of   $C$,  the {\it complementary} curve of $B$  is defined as the closure of $C\setminus B$ and it is denoted by  $B^c$.  We will denote by $\Delta_B$ the Weil divisor $\Delta_B= B \cdot B^c = \sum_{p \in B \cap B^c}p$,  we will  denote its degree by  $\delta_B$ so  $\delta_B=\#B \cap B^c$.  In particular, when $C_i$ is a component of $C$, $\Delta_{C_i}$ is given by the nodes on $C_i$. To simplify notations we set $\Delta_{C_i}= \Delta_i$ and $\delta_i=\#\Delta_i$.
\\
As the only singularities of $C$ are nodes, $C$ can be embedded in a smooth projective surface, see \cite{A79}. This gives, for any  proper subcurve $B$ of $C$, the following fundamental exact sequence
\begin{equation}
\label{EXSEQ:CB}
0\to \OO_{B^c}(-\Delta_B) \to \OO_C \to \OO_B \to 0,
\end{equation}
from which  we deduce
\begin{equation}
\label{EQ:paB}    
p_a(C)=p_a(B)+p_a(B^c)+\delta_B-1.
\end{equation}
We recall that a connected nodal curve $C$  of arithmetic genus $p_a(C) \geq 2$ is called {\it stable} if each smooth rational component $E$  of $C$  meets $E^c$  in at least three points, i.e. $\delta_E \geq 3$. A curve is stable if and only if $\omega_C$ is ample. The curve $C$ is called {\it semistable} if $\delta_E \geq 2$.  If $C$ is semistable, a rational component $E$ with $\delta_E= 2$ is said to be an {\it exceptional component}.   Finally, $C$ is called { \it quasistable} if it is semistable  and if any two exceptional components do not intersect each other. Good references for these topics are \cite{Cap,Cap1}.
\\

Let $L$ be a line bundle on $C$. For all $i = 1,\dots,\gamma$,  let $L_i$ denote  the restriction  of $L$ to the component $C_i$.  It is a line bundle on $C_i$ with  $\deg(L_i) = d_i$.  We will call $(d_1,\dots,d_{\gamma})$ the {\it multidegree} of $L$. Then the {\it degree} of $L$ is $\deg(L) = \sum_{i=1}^{\gamma}d_i$.
We have an exact sequence
$$ 0 \to L \to {\nu}_{*}{\nu}^{*}L \to \bigoplus_{p \in \Sing(C)} {\mathbb C}_p \to  0,$$
from which we deduce 
$\chi(L) = \sum_{i=1}^{\gamma}\chi(L_i) - \delta$.
In complete analogy with the smooth case, 
Riemann-Roch's Theorem  holds for any line bundle $L$ on $C$: $\chi(L) = \deg(L)  + 1 -p_a(C)$. We recall that $L$ is ample if and only if $d_i >0$ for all $i =1,\dots \gamma$. 
We will denote by $\Pic^{\underline 0}(C) \subset \Pic(C)$ the variety parametrizing the isomorphism classes of line bundles on $C$ having multidegree $(0,\dots,0)$. 
\\

There exists on $C$ a dualizing sheaf $\omega_C$, which is invertible. 
For simplicity, if $L$ is a line bundle on $C$ and $B$ is a subcurve of $C$, we will denote by $\deg_B(L)=\deg_B(L|_B)$ the degree of $L|_B$ as line bundle on $B$. Then, we have $\omega_C|_B= \omega_{B}(B\cdot B^c)$, from which we obtain that the degree of $\omega_C|_B$ is $\deg_B(\omega_C|_B)=2p_a(B) -2 + \delta_B$. In particular, we have $\deg(\omega_C) = 2p_a(C) -2$.
\\


A central object in this paper will be the notion of polarization.   One can refer to \cite{OS} and \cite{Ses} for details about polarizations and their role in studying stability of depth one sheaves on reducible nodal curves.

\begin{definition}
\label{DEF:polarization}
A {\bf polarization} on the curve $C$ is a vector  ${\underline w}= (w_1,\dots,w_{\gamma}) \in {\mathbb Q}^{\gamma}$ such that
\begin{equation}
0 < w_i < 1 \quad \sum_{i=1}^{\gamma}w_i = 1.
\end{equation}
We will say that the pair $(C,\w)$ is a 
{\bf polarized curve}. 
\end{definition}
 
\begin{remark} 
\label{linebundlepolarization}
Let $L$ be an ample  line bundle on $C$, with $\deg(L) = d =  \sum_{i=1}^{\gamma}d_i.$
We can associate to $L$ a polarization $\w_L$ on $C$ by setting $\w_L = \frac{1}{d}(d_1,\dots,d_\gamma)$. We will call $\w_L$ the {\bf polarization induced by $L$}.
Note that for any polarization $\w$ there exists a line bundle $L$  which induces $\underline w$. Such a line bundle is not unique: many modifications of $L$ (for instance, one can consider a multiple of $L$), lead to the same polarization. 
\end{remark}

We recall that a {\it depth one sheaf} on a curve is a coherent sheaf $E$ with $\dim Supp(F)= 1$ for any subsheaf $F$ of $E$. On a nodal curve this is equivalent to saying that $E$ is torsion free. If $E$ is a depth one sheaf on $C$ and $B$ is any proper subcurve of $C$, we denote by $E|_B$ the restriction of $E$ to $B$ and by $E_B$ the restriction $E|_B$ modulo torsion. 
Then $E_B$ is a depth one sheaf on $B$. If $C_i$ is an irreducible component of $C$ we define $E_i$ to be $E_{C_i}$. We denote by $d_i$ the degree of $E_i$ and $r_i$ the rank of $E_i$. 
\\

If $\w$ is a polarization on $C$, we define the {\it $\w$-rank} and the {\it $\w$-degree} of $E$ as 
$\wrank(E)=\sum_{i=1}^{r}r_iw_i$ and $\wdeg(E)=\chi(E)-\wrank(E)\chi(\OO_C)$ respectively.

\begin{definition}
Let ${\underline w}$ be a polarization on $C$ and let
$E$ be a depth one sheaf on $C$. 
The {\bf $\w$-slope} of $E$ is defined as
$$\mu_{\w}(E) = \frac{\chi(E)}{\wrank(E)}=  \frac{\wdeg(E)}{\wrank(E)} + \chi(\OO_C).$$
$E$ is said to be {\bf ${\w}$-semistable}   if  for any proper subsheaf $F$ of $E$ we have $\mu_{\w}(F) \leq \mu_{\w}(E)$, i.e. if 
$$\frac{\wdeg(F)}{\wrank(F)} \leq \frac{\wdeg(E)}{\wrank(E)}.$$
$E$ is said to be {\bf $\w$-stable} if the above inequality is strict. 
\end{definition}

We stress that  in the case of depth one sheaves  having rank $1$ on each irreducible component of $C$, many different notions of semistability have been introduced. One can see for instance \cite{E1}  and \cite{OS}, for two different  approaches which give equivalent stability conditions.  
In particular, we recall the following characterization of $\w$-semistability, see \cite{OS}. 

\begin{proposition}
\label{DEF:equiv}
Let $(C,\w)$ be a polarized curve and  let $L$ be a depth one sheaf with $r_i=1$ for all $i$. Then $L$ is $\w$-semistable if and only if for any proper subcurve $B$ of $C$
$$\wdeg(L_B)\geq \wdeg(L) \wrank(L_B).$$
It is $\w$-stable if and only if the inequality is strict.
\end{proposition}

\section{Polarizations with nice properties}
\label{SEC:GOODPOL}

From now on we will assume that $C$  is a reducible nodal curve.

\subsection{The function $\Delta_{\w}$ and its properties}
\label{SUBSEC:DELTA}

\begin{definition}
Let $\w$ be a polarization on $C$. Let $E$ be a depth one sheaf on $C$ and let $E_i$ be the restricion of $E$ to $C_i$ modulo torsion. We define $\Delta_{\w}(E)$ as
$$\Delta_{\w}(E)=\wdeg(E)-\sum_{i=1}^{\gamma} \deg(E_i).$$
\end{definition}

Note that if $p_a(C) = 1$, then $\Delta_{\w}(E) =\chi(E) - \sum_{i=1}^{\gamma}\deg(E_i)$, so it does not depend on the  chosen polarization. 
\\ 

Let $E$ be a depth one sheaf on $C$. 
Let $p\in C_{i_1} \cap C_{i_2}$ be a node of $C$,  then  $\nu^{-1}(p) = \{q_{p,i_1}, q_{p,i_2}\}$ with $q_{p,i_k}\in C_{i_k}$. The stalk of $E$ in $p$ can be written (see \cite{Ses}) as
\begin{equation}
E_p=\OO_p^{s_p}\oplus \OO_{q_{p,i_1}}^{a_{p,i_1}}\oplus \OO_{q_{p,i_2}}^{a_{p,i_2}}
\end{equation}
where $s_{p}$ is the rank of the free part of the stalk of $E$ in $p$. Moreover, we have
\begin{equation}
\label{EQ:RSA}
r_{i_1}=s_p+a_{p,i_1} \qquad r_{i_2}=s_p+a_{p,i_2}.
\end{equation} 
We set $t_{p}=a_{p,i_1}+a_{p,i_2}$.

\begin{definition}
\label{DEF:RESIDUALRANK}
Let  $E$ be a   depth one  sheaf on $C$ and  let $p$ be a node,  we will call  $t_p$ the {\bf residual rank of $E$ at $p$}.
\end{definition}

\begin{remark}
\label{REM:RESIDUALZERO}
A depth one sheaf $E$ on $C$ is locally free if and only all the residual ranks of $E$ are zero.
\end{remark}

In the following lemma we summarize some basic properties satisfied by $\Delta_{\w}$.
\begin{lemma}
\label{LEM:POLGEN}
Let $C$ be a nodal curve with nodes $p_1,\dots, p_{\delta}$ and let  $\w$ be a polarization on it. Let  $E$ be a depth one sheaf on $C$. Then we have: 
\begin{enumerate}[(a)]
\item set $\lambda_i=\Delta_{\w}(\OO_{C_i})=1-g_i-w_i\chi(\OO_C)$. Then  $\sum_{i=1}^{\gamma}\lambda_i=\delta$ and 
\begin{equation}
\label{EQ:DELTAGENERAL}
\Delta_{\w}(E)=\sum_{i=1}^{\gamma}  r_i \lambda_i  - \sum_{j=1}^{\delta}s_{p_j};
\end{equation}
\item if $E$ is locally free, then $\Delta_{\w}(E)=0$, i.e. $\wdeg(E) = \sum_{i=1}^{\gamma}\deg(E_i)$;
\item if $r_i=r$ for all $i=1,\dots,\gamma$, then for any node $p_j$ we have  $a_{p_j,i_1}=a_{p_j,i_2}=t_{p_j}/2$. Moreover, $\Delta_{\w}(E)=\frac{1}{2}\sum_{j=1}^{\delta}t_{p_j}\geq 0$ and equality holds if and only if $E$ is locally free;
\item{} for any line bundle  $L$  we have 
$\Delta_{\w}(E \otimes L) = \Delta_{\w}(E)$; 
\item{} if $\supp(E)$  is a disjoint union of connected subcurves $B_s$ for $s=1,\dots, c$,  then 
$$ \Delta_{\w}(E) = \sum_{s=1}^c\Delta_{\w}(E|_{B_s});$$
\item{}  if $B$ is a   proper subcurve of $C$,  then     $\Delta_{\w}(E_B)+ \Delta_{\w}(E_{B^c}) = \Delta_{\w}(E) + \sum_{p_j \in B \cap B^c}s_{p_j}$;
\item if $E$ is locally free of rank $r$ and $B$ is a subcurve of $C$,  then  $\Delta_{\w}(E|_B)=r\Delta_{\w}(\OO_B).$
\end{enumerate}
\end{lemma}

\begin{proof}
(a) From \cite{Ses} we have an exact sequence
\begin{equation}
\label{EQ:Seshadri}
0\to E \to  {\bigoplus}_{i=1}^{\gamma} E_i \to T \to 0,
\end{equation}
where $T$ is a torsion sheaf on $C$ whose support is contained in the set of nodes. Hence  we have $\chi(E)= \sum_{i=1}^{\gamma}\chi(E_i) - \chi (T)$. More precisely, if $p_j$ is a node, we have $h^0(T_{p_j})=s_{p_j}$ so $\chi(T)=\sum_{j=1}^{\delta}s_{p_j}$.
Then, by definition, we have 
\begin{multline*}
\label{EQ:wdeg1}
\wdeg(E)= \sum_{i=1}^{\gamma}\chi(E_i) - 
\chi(T) -\wrank(E) \chi(\OO_C)= \sum_{i=1}^{\gamma}[d_i + r_i(1-g_i)] - \sum_{i=1}^{\gamma}w_ir_i\chi(\OO_C) - \sum_{j=1}^{\delta}s_{p_j}
\end{multline*}
so we get
$$\Delta_{\w}(E)=\sum_{i=1}^{\gamma}r_i[1-g_i-w_i\chi(\OO_C)] - \sum_{j=1}^{\delta}s_{p_j} = \sum_{i=1}^{\gamma}r_i \lambda_i  - \sum_{j=1}^{\delta}s_{p_j}.$$
 Finally, we have $$\sum_{i=1}^{\gamma} \lambda_i= \sum_{i=1}^{\gamma}[1-g_i-w_i\chi(\OO_C)]=\gamma-(p_a(C)-\delta+\gamma-1)-(1-p_a(C)) =\delta.$$

(b) Let $E$ be a locally free sheaf of rank $r$. By the previous formula we have
$$\Delta_{\underline{w}}(E)=\sum_{i=1}^{\gamma}r\lambda_i-\sum_{j=1}^{\delta}r=r\sum_{i=1}^{\gamma}\lambda_i-{\delta}r=r\delta-r\delta=0.$$

(c) Assume that $r_i = r$ for all $i=1,\dots,\gamma$.  By Equation \eqref{EQ:RSA} we get  $a_{p_j,i_1}=a_{p_j,i_2}:=a_j$.  As $s_{p_j} = r- a_j$, from (a)  we have
$$\Delta_{w}(E)=r \sum_{i=1}^{\gamma} {\lambda}_i - \sum_{j=1}^{\delta} (r - a_j) =r\delta- r \delta +  \sum_{j=1}^{\delta} a_j = \sum_{j=1}^{\delta}a_j= \frac{1}{2} \sum_{j=1}^{\delta}t_j.$$
Then $\Delta_{\w}(E)=0$ if and only if $t_j = 0$  for all $j$, that is $E$ is locally free.
\\

(d) Let $L$ be a line bundle on $C$ with $\deg(L_i) = l_i$, $i=1,\dots,\gamma$. Since $(E \otimes L)_i = E_i \otimes L_i$, we have
$\deg(E \otimes L)_i= d_i + r_il_i$ and $\chi(E_i \otimes L_i) = \chi(E_i) + r_il_i$. 
If we tensor the  exact sequence \eqref{EQ:Seshadri} by $L$ we obtain $\chi(E \otimes L) = \chi(E) - \sum_{i=1}^{\gamma} r_il_i$, hence we get
$$\Delta_{\w}(E \otimes L) 
= \chi(E \otimes L) - \wrank(E \otimes L) \chi(\OO_C) - \sum_{i=1}^{\gamma} (d_i + r_il_i)= $$
$$ = \chi(E) + \sum_{i=1}^{\gamma}r_il_i - \wrank(E) \chi(\OO_C) - \sum_{i=1}^{\gamma} (d_i + r_il_i)= \Delta_{\w}(E).$$ 

(e) Let   $B = \sqcup_{s=1}^c B_s$ be  the disjoint union of   $c \geq 1$ connected proper subcurves $B_s$. 
Since  $\Supp(E) = B$, then $E =  \oplus_{s=1}^cE_{B_s}$ and  $E_{B_s}$  is a depth one sheaf too. The $\w$-degree is additive with respect to direct sum,  so  we have
$$\Delta_{\w}(E)= \wdeg(E) - \sum_{C_{i}\subseteq B}d_i  =  \sum_{s= 1}^{c} \wdeg(E_{B_s}) - \sum_{s= 1}^c \sum_{C_i \subseteq B_s} d_i = \sum_{s=1}^c \Delta_{\w}(E_{B_s}).$$  
(f) Assume that $B$ is a proper connected curve.
By  (a)   we have:
$$\Delta_{\w}(E_B)= \sum_{C_i \subseteq B}r_i \lambda_i - \sum_{p_j \in B \setminus B^c}s_{p_j}.$$
If $B = \sqcup_{s=1}^c B_s$ is   the disjoint union of   $c \geq 1$ connected proper subcurves $B_s$. 
Then by (e) 
$$\Delta_{\w}(E_B)= \sum_{s=1}^c \Delta_{\w}(E_{B_s}) = \sum_{s= 1}^c \left[\sum_{C_i \subseteq B_s} r_i\lambda_i - \sum_{ p_j \in B_s \setminus B_s^c}s_{p_j}\right]= \sum_{C_i\subseteq B} r_i\lambda_i - \sum_{p_j \in B \setminus B^c}s_{p_j}.$$
and a similar formula holds for $B^c$. So we have:
$$\Delta_{\w}(E_B) + \Delta_{\w}(E_{B^c})= 
\sum_{C_i \subseteq B}r_i \lambda_i - \sum_{p_j \in B \setminus B^c}s_{p_j} + \sum_{C_i \subseteq B^c}r_i \lambda_i - \sum_{p_j \in B^c \setminus B}s_{p_j}=
\sum_{i=1}^{\gamma}r_i \lambda_i- \sum_{p_j \not\in B \cap B^c}s_{p_j}.$$
As $\Delta_{\w}(E) = \sum_{i=1}^{\gamma}r_i\lambda_i - \sum_{j=1}^{\delta}s_{p_j},$ we obtain:
$$\Delta_{\w}(E_B) + \Delta_{\w}(E_{B^c})= \Delta_{\w}(E) + \sum_{p_j \in B \cap B^c}s_{p_j}.$$
(g) By (e) it is enough to prove the assertion for  any  connected  subcurve $B$. Then  $ B = \bigcup_{k=1}^{b}C_{i_k}$. 
Then
\begin{multline*}
\Delta_{\w}(E|_B)= \wdeg(E|_B) - \sum_{k=1}^bd_{i_k} = \chi(E|_B) - \wrank(E|_B) \chi(\OO_C) - \sum_{k=1}^{b} d_{i_k}=\\ = \sum_{k=1}^{b}d_{i_k} + r(1-p_a(B)) - r\sum_{k=1}^b w_{i_k} \chi(\OO_C) - \sum_{k=1}^b d_{i_k}  = r \wdeg(\OO_B)  = r \Delta_{\w}(\OO_B).
\end{multline*}
as claimed.
\end{proof}

The following proposition gives a description of $\Delta_{\w}(E)$ as function of the residual ranks of $E$ at its nodes and its multirank. 

\begin{proposition}
\label{PROP:DELTARESIDUAL}
Let $(C,\w)$ be a connected nodal polarized curve. Let  $E$ be a depth one sheaf on $C$, then we have
$$\Delta_{\w}(E)=\sum_{i=1}^{\gamma}r_i\left(\lambda_i - \frac{\delta_i}{2}\right)+\frac{1}{2}\sum_{j=1}^{\delta} t_{p_j},$$
where $t_j$ is the residual rank of $E$  at the node $p_j$.

\end{proposition}

\begin{proof}
Let $p_j$ be a node, assume that $ p_j \in C_{i_1} \cap C_{i_2}$. To avoid confusion we  denote by $r_{j,i_k}$  the rank of $E|_{C_{i_k}}$ so that 
$$r_{j,i_1}+r_{j,i_2}=2s_{p_j}+a_{p_j,i_1}+a_{p_j,i_2}=2s_{p_j}+t_{p_j}.$$
We recall that we set $\lambda_i=1-g_i-w_i\chi(\OO_C)$,  so by Equation \eqref{EQ:DELTAGENERAL} we have
\begin{equation}
\label{EQ:DELTAPARTIAL}
\Delta_{\w}(E)=
\sum_{i=1}^{\gamma}\lambda_ir_i-\sum_{j=1}^{\delta}\frac{r_{j,i_1}+r_{j,i_2}-t_{p_j}}{2}= 
\sum_{i=1}^{\gamma}\lambda_ir_i-\frac{1}{2}\sum_{j=1}^{\delta}(r_{j,i_1}+r_{j,i_2})+\frac{1}{2}\sum_{j=1}^{\delta}t_{p_j}.
\end{equation}
We claim that the following relation holds:
\begin{equation}
\label{EQ:SWITCH}
\sum_{j=1}^{\delta}(r_{j,i_1}+r_{j,i_2})=\sum_{i=1}^{\gamma}r_i\delta_i
\end{equation}
We will proceed by induction on $\gamma$. If $C$ has $2$ components and $\delta$ nodes,  we denote by $r_1$ and $r_2$ the ranks of the restrictions of $E$ to the components. If $p_j$ is a node, then $r_{j,i_1}+r_{j,i_2}=r_1+r_2$ so 
$$\sum_{j=1}^{\delta}(r_{j,i_1}+r_{j,i_2})=\sum_{j=1}^{\delta}(r_1+r_2)=\delta r_1+\delta r_2$$
so  Equation \eqref{EQ:SWITCH} holds when $\gamma=2$.
\\

Assume now, by induction hypothesis, that the same equation holds for nodal curves with  at most $\gamma-1$ components. Let $C$  be a nodal curve with  $\gamma$ components. We claim that there exists a component of $C$ whose complementary curve is connected. 
This is true since the graph $\Gamma_C$ is connected and every connected graph has a non-disconnecting vertex\footnote{Let $G=(\cV,\cE)$ be a finite connected graph, with at least $3$ vertices. Then one can fix $P\in \cV$ and consider the distance $d_P(Q)$ of $Q$ from $P$, i.e. the minimum number of edges that one needs to go through in order to make a path from $P$ to $Q$. Let $R\in \cV$ such that $d_P(R)=\max_{Q\in \cV}d_P(Q)$. Then $R$ is a non-disconnecting vertex of $G$. Indeed, if $Q\in \cV$ different from $R$, the shorthest path from $P$ to $Q$ cannot pass through $R$ otherwise $d_P(R)<d_P(Q)$ and we get a contradiction.}. 
Fix an ordering of the components of $C$ in such a way that this non-disconnecting curve is $C_\gamma$. By assumption, its complementary curve $C_{\gamma}^c$ is connected, with $\gamma'=\gamma-1$ components with indices $i=1,\dots,\gamma-1$. Moreover, it has $\delta'=\delta-\delta_\gamma$ nodes and $r_i'=r_i$ for all $i=1,\dots,\gamma-1$.
We can write
\begin{equation}
\label{EQ:SPLIT}
\sum_{j=1}^{\delta}(r_{j,i_1}+r_{j,i_2})=\sum_{p_j\not\in C_\gamma}(r_{j,i_1}+r_{j,i_2})+\sum_{p_j\in C_\gamma}(r_{j,i_1}+r_{j,i_2}).
\end{equation}
In the first summation on the right hand side of Equation \eqref{EQ:SPLIT}, the sum is done over the nodes which are not on $C_\gamma$ so they are exactly the nodes of $C_\gamma^c$ as  a nodal curve. Then, by induction hypothesis, we have
$$\sum_{p_j\not\in C_\gamma}(r_{j,i_1}+r_{j,i_2})=\sum_{j=1}^{\delta'}(r_{j,i_1}+r_{j,i_2})=\sum_{i=1}^{\gamma'}(r_i'\delta_i').$$
For all  $i=1,\dots,\gamma-1$ we  denote by $\epsilon_i$ the number of points of $C_i\cap C_\gamma$, i.e. the nodes common to $C_i$ and $C_\gamma$. Then we have   $\delta'_i=\delta_i-\epsilon_i$, as
the nodes of $C_i \cap C_{\gamma}$ are not nodes of 
$C_{\gamma}^c.$
If $\epsilon_i=0$,   $C_i$ and $C_\gamma$ are disjoint  and  $\delta'_i=\delta_i$.  So we have: 
\begin{equation}
\label{EQ:SPLIT1}
\sum_{p_j\not\in C_\gamma}(r_{j,i_1}+r_{j,i_2})=\sum_{i=1}^{\gamma-1}r_i(\delta_i-\epsilon_i).
\end{equation}

In the second summation on the right hand side of Equation \eqref{EQ:SPLIT}, the sum is done over the $\delta_\gamma$ nodes which are on $C_\gamma$ so we can write
$$\sum_{p_j\in C_\gamma}(r_{j,i_1}+r_{j,i_2})=\sum_{C_i\,|\, C_i\cap C\gamma\neq \emptyset}(r_{\gamma}+\epsilon_i r_i)$$
as $r_i=r_{j,i_k}$ for some $j$ if and only if $C_i$ is one of the components  through $p_j$ and this happens one times for each of the nodes which are on both $C_i$ and $C_\gamma$, i.e. exactly $\epsilon_i$ times. Hence
\begin{equation}
\label{EQ:SPLIT2}
\sum_{p_j\in C_\gamma}(r_{j,i_1}+r_{j,i_2})=[\cdots]=\sum_{C_i\,|\, \epsilon_i>0}r_{\gamma}+\sum_{C_i\,|\, \epsilon_i>0}\epsilon_i r_i+\sum_{C_i\,|\, \epsilon_i=0}\epsilon_i r_i=r_{\gamma}\delta_\gamma +\sum_{i=1}^{\gamma-1}\epsilon_i r_i.
\end{equation}
Then, using Equations \eqref{EQ:SPLIT1} and \eqref{EQ:SPLIT2}, we can rewrite Equation \eqref{EQ:SPLIT} as
$$
\sum_{j=1}^{\delta}(r_{j,i_1}+r_{j,i_2})=\sum_{i=1}^{\gamma-1}r_i(\delta_i-\epsilon_i)+r_{\gamma}\delta_\gamma +\sum_{i=1}^{\gamma-1}\epsilon_i r_i=\sum_{i=1}^{\gamma}r_i\delta_i
$$
which concludes the proof of the claim. From Equations \eqref{EQ:DELTAPARTIAL} and \eqref{EQ:SWITCH} one obtains easily the desired result.
\end{proof}

\subsection{Good polarizations and main properties}
\label{SUBSEC:GOODPOL}

~\\ 
Now we will deal with a class of polarizations which will allow us to extend some properties that hold for locally free sheaves on smooth curves to depth one sheaves on polarized  nodal curves (see Theorem \ref{StableOC}). In order to do this we will use the function $\Delta_{\w}$ that we have studied in Subsection \ref{SUBSEC:DELTA}.
 
\begin{definition}
\label{GOOD}
Let $(C,\w)$ be a polarized nodal curve. We say that $\w$ is a {\bf good polarization} if $\Delta_{\w}(E)\geq 0$ for all depth one sheaves  $E$ on $C$ and equality holds if and only if $E$ is locally free. 
\end{definition}

By Lemma \ref{LEM:POLGEN} (b), for any polarization $\underline w$ we have $\Delta_{\w}(E) = 0$  for all locally free sheaves on $C$.  Nevertheless, it can happen that $\Delta_{\w}(E)<0$ for a depth one sheaf which is not locally free, as the next example shows.

\begin{example}
Let $(C,\w)$ be a polarized nodal curve with two smooth components $C_1$ and $C_2$ of genus $2$ and a single node. Then $\Delta_{\w}(\OO_{C_1})=-1+3w_1$.  If we consider the polarization  $\w=\left(\frac{1}{6},\frac{5}{6}\right)$, we have $\Delta_{\w}(\OO_{C_1})=-1/2 < 0$. Moreover, by Proposition \ref{DEF:equiv}, this also implies that $\OO_C$ is a $\w$-unstable sheaf on $C$.
\end{example}

First of all we will see that on all stable nodal curves with $p_a(C)\geq 2$ there exists a good polarization (we will see in Remark \ref{REM:cycle} that is not true in general).

\begin{proposition}
\label{polarizationeta}
Let $C$ be a stable connected nodal curve with $p_a(C)\geq 2$ and let be  $\underline{\eta}$ be the polarization induced by $\omega_C$ (this is often called {\it canonical polarization}). Then, $\underline{\eta}$ is a good polarization on $C$.
\end{proposition}

\begin{proof}
First of all, since $C$ is a stable curve, we have that $\omega_C$ is an ample line bundle so the definition of $\underline{\eta}$ makes sense. 
As recalled in Section \ref{SEC:NOTATIONS} we have
$\omega_C|_{C_i}=\omega_{C_i}\otimes (\Delta_i)$ so we have
$$\eta_i= \frac{g_i-1+\delta_i/2}{p_a(C)-1}, \quad i = 1,\dots \gamma.$$


In order to see that $\underline{\eta}$ is good we will compute
$\Delta_{\underline{\eta}}(E)$ for a depth one sheaf $E$.
For the canonical polarization we have $${\lambda}_i = 1-g_i-\eta_i\chi(\OO_C)=\delta_i/2,$$ so, by Proposition \ref{PROP:DELTARESIDUAL} we can conclude that 
$$\Delta_{\underline{\eta}}(E)=\frac{1}{2}\sum_{j=1}^{\delta}t_{p_j}.$$
In particular, $\Delta_{\underline{\eta}}(E)\geq 0$ and equality holds if and only if $t_{p_j}=0$ for all $j$. By Remark \ref{REM:RESIDUALZERO}, this happens if and only if $E$ is locally free.
\end{proof}

The following theorem summarizes some important properties which hold when we deal with good polarizations. Recall that $\Pic^{\underline{0}}(C)$ is the variety parametrizing line bundles having degree $0$ on each component (see Section \ref{SEC:NOTATIONS}).

\begin{theorem}
\label{StableOC}
Let $C$ be a nodal curve and $\w$ a good  polarization on it. Let $E$ be a depth one sheaf on $C$. Then we have the following properties:  
\begin{enumerate}[(a)]
\item{} Assume that $E$ is locally free and,  for $i=1,\dots,\gamma$, $E_i$ is stable with $\deg(E_i)= 0$.  Then $E$ is $\w$-stable.  
\item{} If $E$ is globally generated, then $\wdeg(E) \geq 0$. 
\item{} If $E$ is $\w$-semistable and $\wdeg(E)>0$,  then $h^0(E^*)=0$. 
\end{enumerate}
In particular, if $E=\OO_C$ or more generally $E \in \Pic^{\underline 0}(C)$ then, $E$ is $\w$-stable.
\end{theorem}

\begin{proof}

(a) Let $E$ be a locally free sheaf such that $E_i$ is stable and $\deg(E_i)=0$ for all $i=1,\dots,\gamma$.  Then, by Lemma \ref{LEM:POLGEN}(b) we have $\wdeg(E) =0$. In order to prove that $E$ is $\w$-stable it is enough to show that for any proper subsheaf $F$ of $E$ we have $\wdeg(F)<0$.
\\

Let $F$ be a proper subsheaf of $E$ and let's consider the quotient $Q=E/F$. If $\wrank(Q)=0$, then $Q$ is a torsion sheaf with finite support. Then $\wdeg(Q)=\sum_{P\in \Supp(Q)}l(Q_P)>0$ and then $\wdeg(F)<0$ as claimed.
\\

Assume now that $\wrank(Q)>0$. Since $F$ is a proper subsheaf of $E$ we also have $\wrank(Q)<\wrank(E)$. We define $Q'=Q/\Tors(Q)$ which  is a depth one sheaf with $\wrank(Q')= \wrank(Q)$ and $\wdeg(Q) \geq \wdeg(Q')$.
Moreover, as $Q'$ is a quotient of $Q$ we have that $Q'$ is a   proper quotient of $E$.  So for all $i=1,\dots,\gamma$,  we have a  surjective map  
$q_i: E_i\to Q_i'$.
If $Q_i'$ is not zero, then either $q_i$ is an isomorphism (this cannot occur for all  $i$) or $Q_i'$ is a proper quotient, in this case $\deg(Q_i')>\deg(E_i)=0$ by the stability assumption on $E_i$.
Hence $\sum_i \deg(Q_i')>0$. Then, as $\w$ is a good polarization, we have 
$$\Delta_{\w}(Q')=\wdeg(Q') - \sum_i\deg(Q_i')\geq 0$$
which implies $\wdeg(Q')>0$. Then $\wdeg(Q)>0$ and we can conclude as in the previous case.
\\

From $(a)$, if $E$ is a line bundle with $\deg(E_i)=0$ for all $i=1,\dots,\gamma$, we have that $E$ is $\w$-stable. One can also prove this fact directly using Proposition \ref{DEF:equiv} by checking that $\wdeg(L|_B) > 0$ for any proper subcurve $B$.
Indeed, we have
$$\wdeg(L|_B)= \Delta_{\w}(L|_B) >0,$$
as $\w$ is a good polarization and $L|_B$ is not locally free on $C$. \\

(b) Assume that $E$ is a depth one sheaf on $C$ which is generated by $k \geq 1$ global sections. Then we have a surjective map
$V \otimes \OO_C \to E$,
where $V \subseteq H^0(E)$ is a vector space of dimension $k$.
Since by (a), $\OO_C$ is $\w$-stable, then $V \otimes \OO_C$ is $\w$-semistable.  So we have $\frac{\wdeg(E)}{\wrank(E)} \geq 0$ and then $\wdeg(E)\geq 0$.
\\

(c) Assume that $H^0(E^*)=\Hom(E,\OO_C)\neq 0$. Then, there exists a non zero homomorphism $\varphi \colon E \to \OO_C$. We will show that $\wdeg(E)<0$. If $\varphi$ is surjective or injective, we conclude by $\w$-semistability of $E$ and by $\w$-stability of $\OO_C$ (which holds by (a), since $\w$ is good) respectively. We can assume then, that $\Img(\varphi)$ is a proper subsheaf of $\OO_C$ and a proper quotient of $E$. In this case we have
$$\wdeg(E)/\wrank(E) \leq \wdeg(\Img(\varphi))/\wrank(\Img(\varphi))<0$$
where we used the $\w$-semistability of $E$ and the $\w$-stability of $\OO_C$ respectively.
\end{proof}

\begin{remark}
In point (a) of Theorem \ref{StableOC} if $E_i$ is only semistable then, with the same arguments, one obtain that $E$ is $\w$-semistable.
\end{remark}

Another interesting consequence of the previous theorem is the following corollary.

\begin{corollary}
Let $C$ be a nodal curve and  $\w$ a good polarization. 
If $\w=\w_L$ for some ample line bundle $L$, $\w$-(semi)stability is preserved by tensoring with $L$. In particular, $L$ is $\w$-stable.
\end{corollary}

\begin{proof}
Let $L$  be a line bundle which induces the polarization ${\underline w}_L$, with $L_i \in \Pic^{d_i}(C_i)$. Since  $w_i=d_i/d$ then we have
$$ w_i d_j = w_j d_i.$$
This implies, by \cite{Ses}, that ${\underline w}_L$-stability is preserved  by tensoring with $L$. In particular, since  $\OO_C$ is ${\underline w}_L$-stable by Theorem \ref{StableOC}, then $L$ is ${\underline w}_L$-stable too.
\end{proof}


\subsection{Polarizations and $\w$-stability of $\OO_C$}

~\\ 
In this subsection we investigate polarized nodal curves $(C,\w)$ with $\w$-stable $\OO_C$.

\begin{lemma}
\label{LEM:BOUNDS}
Let $(C,\w)$ be a polarized nodal curve. Then $\OO_C$ is $\w$-stable if and only if \begin{equation}
\label{EQ:STABOCDELTA}
0<\Delta_{\w}(\OO_B)<\delta_B
\end{equation} for any proper subcurve $B$ of $C$. If equality holds for some subcurve $B$ then $\OO_C$ is $\w$-semistable. 
Moreover we can specialize the result in the following cases:
\begin{description}
\item [$\bullet\, p_a(C)=0$] $\OO_C$ is always $\w$-stable;
\item [$\bullet\, p_a(C)=1$] $\OO_C$ is always $\w$-semistable and it is $\w$-stable if and only if $C$ is a cycle of rational curves;
\item [$\bullet\, p_a(C)\geq 2$] $\OO_C$ is $\w$-stable if and only the conditions
\begin{equation}
\label{EQ:DEFSTARB}
(\star)_B :\qquad \frac{p_a(B) -1}{p_a(C)-1} < \wrank(\OO_B) < \frac{p_a(B)-1+\delta_B}{p_a(C) -1}
\end{equation}
hold for all proper subcurves $B$ of $C$. 
\end{description}
Actually, it is enough to check the the Inequalities \eqref{EQ:STABOCDELTA} and \eqref{EQ:DEFSTARB} only for connected subcurves.
\end{lemma}

\begin{proof}
By Proposition \ref{DEF:equiv} we have that  $\OO_C$ is $\w$-stable if and only if $ \wdeg(\OO_B) > 0$ for any proper subcurve $B$ of $C$. Moreover, 
since $\wdeg(\OO_B) = \Delta_{\w}(\OO_B)$, by Lemma \ref{LEM:POLGEN}(e), it is enough to check the condition $\wdeg(\OO_B)>0$ only for connected subcurves.
\\

Let $B$ be a proper subcurve of $C$ and $B^c$ its complementary curve. 
Then $\OO_B$ and $\OO_{B^c}$ are two depth one sheaves on $C$. We have 
$$\wdeg(\OO_{B})= \chi(\OO_{B}) - \wrank(\OO_{B}) \chi(\OO_C) =  1-p_a(B)- \wrank(\OO_B)\chi(\OO_C).$$
From Equation \eqref{EXSEQ:CB} we have 
$\chi(\OO_{B^c}) = \chi(\OO_C)- \chi(\OO_B) +\delta_B$, so
$$
\wdeg(\OO_{B^c})=\chi(\OO_{B^c}) - \wrank(\OO_{B^c}) \chi(\OO_C) =$$
$$ = \chi(\OO_C)  - \chi(\OO_B) + \delta_B -
(1- \wrank(\OO_B))\chi(\OO_C)=\wrank(\OO_B)\chi(\OO_C) + p_a(B) -1 + \delta_B.$$ 
Hence $\OO_C$ is $\w$-stable if and only if  both the above values are  strictly positive, we obtain Inequality \eqref{EQ:STABOCDELTA}. If $p_a(C) \geq 2$, solving the inequalities we get  condition $(\star)_B$.
\\

Assume now $p_a(C)=0$. Then $C$ is a curve of compact type whose components are rational. Then, if $B$ is a proper connected subcurve of $C$, we have that $B$ is also of compact type. In particular $p_a(B)=0$ too. By Inequality \eqref{EQ:STABOCDELTA} we get $1 - \delta_B < \wrank(\OO_B) < 1$, so $\OO_C$ is $\w$-stable.
\\

Assume now $p_a(C)=1$. Then Inequality \eqref{EQ:STABOCDELTA} is equivalent to $1 - \delta_B < p_a(B) < 1$. Since $p_a(B)\leq 1$ and $p_a(B)\geq 1-\delta_B$ we have that $\OO_C$ is always $\w$-semistable. Now we investigate the $\w$-stability of $\OO_C$.
As $p_a(C)=1$, we have either $C$ is of compact type whose components consist of an elliptic curve $C_1$ and $\gamma-1$ rational curves or the dual graph has a single cycle and all components are rationals. 
In the first case, $p_a(C_1)=1$ so $\OO_C$ is never $\w$-stable. In the second case, if we can find a proper connected subcurve $B$ of $C$ which contains a cycle then $p_a(B)=1$ and $\OO_C$ is never $\w$-stable. This happens exactly when $C$ is not a cycle. If $C$ is a cycle and $B$ is a proper connected subcurve, then $\delta_B=2$ and $p_a(B)=0$ so $\OO_C$ is $\w$-stable.
\end{proof}

\begin{remark}
\label{REM:cycle}
Let $C$ be a nodal curve with $p_a(C)=1$ which is not a cycle. Then good polarizations do not exist on $C$.  
\end{remark}

\begin{remark}
\label{REM:TEIXIDOR}
Assume that $(C,\w)$ is a polarized nodal curve of compact type. We can translate the conditions of $\w$-stability for $\OO_C$, given by Teixidor i Bigas in \cite{Tei91}, using our notation as follows:
$\OO_C$ is $\w$-stable if and only if $0<\Delta_{\w}(\OO_{A_i})<1$ for a suitable family of connected subcurves $A_i \subset C$. 
\end{remark}

\begin{corollary}
\label{COR:necbounds}
Let ${\underline w}$ be a good polarization on a nodal curve $C$ with $p_a(C) \geq 2$.  Then $\w$ satisfies $(\star)_B$ for all $B$ subcurve of $C$. In particular, we have
$$\frac{g_i-1}{p_a(C)-1} < w_i < \frac{g_i-1+\delta_i}{p_a(C)-1}.$$
\end{corollary}

An interesting question is then the following:

\begin{question}
\label{QUE:GOODANDOCSTAB}
Are all polarizations for which $\OO_C$ is $\w$-stable also good?
\end{question}

We will give a complete answer for curves of compact type in Section \ref{SEC:3}.

\subsection{Polarizations and balanced line bundles}
\label{SUBSEC:BALANCED}

~\\ 
In this subsection we deal with polarized curves $(C,\w_L)$ where $\w_L$ is induced by a line bundle $L$. We highlight the relation between the $\w_L$-stability of $\OO_C$ and a particular class of line bundles: balanced line bundles (for  details one can see \cite{Cap1,Cap}).

\begin{definition}
Let $C$ be a quasistable curve of  arithmetic genus $p_a(C) \geq 2$. A line bundle  $L$ on $C$ is said to be {\bf balanced}  if the following  properties hold:
\hfill\par
\begin{enumerate}
\item{} for every exceptional component $E$ of $C$  we have ${\deg}_E( L )= 1$;
\item{} for any  proper subcurve $B$  we have
\begin{equation}
\label{EQ:BALANCED}
\left\vert {\deg}_B(L) - \frac{\deg(L)}{2p_a(C)-2} {\deg}_B(\omega_C) \right\vert \leq \frac{1}{2}\delta_B.
\end{equation}
\end{enumerate}
$L$ is said to be {\bf strictly balanced} if 
 the  inequality  is strict  for every subcurve $B$ such that $B  \cap  B^c $ is not contained in the exceptional locus of $C$.
\end{definition}

\begin{proposition}
Let $C$ be a quasistable nodal curve with $p_a(C) \geq 2$.  
Let $L \in \Pic^d(C)$ be an ample line bundle and 
let $ \w = {\underline w}_L$  be the polarization induced by $L$. 
\begin{enumerate}[(a)]
\item{} If $d \geq  p_a(C) -1$ and $L$ is balanced, then 
$\OO_C$ is $\w$-semistable and it is $\w$-stable when $d > p_a(C) -1$;
\item{} if $d \leq  p_a(C) -1$ and $\OO_C$ is $\w$-stable then  $C$ is stable and $L$ is  strictly balanced. 
\end{enumerate}
\end{proposition}
\begin{proof}
Let $L \in \Pic^d(C)$ be an ample line bundle. Then $d_i = \deg(L_i) > 0$ for all $i$ and $d = \sum_{i=1}^{\gamma}d_i$.
As $\w$ is induced by $L$,  we have $w_i = \frac{d_i}{d}$, for all $i=1,\dots \gamma$.
Let $B$ be a subcurve of $C$. Then $B= \bigcup_{k=1}^{b}C_{i_k}$. 
Since $L|_B$ is a line bundle on $B$,  we have:
$$\deg_B(L) = \sum_{k=1}^{b}d_{i_k}=  \sum_{k=1}^bw_{i_k}d =  d \wrank(\OO_B),$$
moreover we recall that 
$$\deg_B(\omega_C) = 2p_a(B)-2 + \delta_B.$$
 We have: 
\begin{equation}
\label{mainEQ}
\left\vert \deg_B(L) - \frac{d}{2p_a(C)-2} \deg_B(\omega_C) \right\vert = \left\vert   d \wrank(\OO_B) -  \frac{d}{p_a(C) -1} (p_a(B) -1 + \delta_B/2) \right\vert =  
\end{equation}
$$ = \frac{d}{p_a(C) -1} \left\vert (p_a(C)-1) \wrank(\OO_B) - (p_a(B) -1 +\delta_B/2) \right\vert.$$
Note that condition $(\star)_B$ in Lemma \ref{LEM:BOUNDS} can be also written as
$$p_a(B) -1 < (p_a(C) -1) \wrank (\OO_B) < p_a(B) -1 + \delta_B, $$
which is equivalent to
$$ \left\vert (p_a(C)-1) \wrank(\OO_B) - (p_a(B) -1 + \delta_B/2) \right\vert < \delta_B /2.$$

(a) Let $d \geq p_a(C) -1$  and assume that $L$ is balanced.
Then Equations \eqref{EQ:BALANCED} and \eqref{mainEQ} imply
$$\left\vert (p_a(C) -1)  \wrank(\OO_B) - (p_a(B) -1
+ \delta_B /2) \right\vert  \leq  \frac{\delta_B}{2}  \frac{p_a(C) -1}{d}.$$
If  $d > p_a(C)-1$, we get
$$ \left\vert (p_a(C) -1) \wrank(\OO_B)- p_a(B) +1 - \delta_B /2 \right\vert < \delta_B /2, $$
which is equivalent to  $(\star)_B$.  This implies that $\OO_C$ is $\w$-stable. 
If $d= p_a(C)-1$,  we get  $$ \left\vert (p_a(C) -1) \wrank(\OO_B)- p_a(B) +1 - \delta_B /2 \right\vert \leq  \delta_B /2,$$
so we can conclude that  $\OO_C$ is $\w$-semistable.
\\

(b)  Let $d \leq p_a(C) -1$ and assume  that $\OO_C$ is $\w$-stable. Then  ${\underline w}$ satisfies $(\star)_B$ for all subcurve $B$.  Let $R$ be a rational component of $C$,  since  $(\star)_{R}$ holds, we have:
$$-1 < (p_a(C) -1) w_R < \delta_R -1.$$
We recall that  $w_R = \frac{d_R}{d}$ and $d_R \geq 1$ since $L$ is ample.  So  we have:
$$  1 \leq d_R < \frac{d}{p_a(C)-1}(\delta_R -1),$$
as  $d \leq  p_a(C) -1$ we  obtain
$1 \leq  d_R  < \delta_R -1$. This implies $\delta_R \geq 3$, so $C$ is a stable curve. 
\hfill\par
Now we prove that $L$ is strictly balanced. Since $d \leq  p_a(C) -1$ we have
 $$ \left\vert \deg_B(L) - \frac{d}{2p_a(C)-2} \deg_B({{\omega}_C}) \right\vert < \frac{d}{(p_a(C) -1)}  \frac{\delta_B}{2} \leq  \frac{\delta_B}{2}
$$
by Inequality \eqref{mainEQ}. This proves that $L$ is strictly balanced. 
 \\
 
 \end{proof}
 
\begin{corollary}
\label{COR:balanced}
Let $C$ be a stable nodal curve with $p_a(C) \geq 2$. Let $L$ be an ample  line bundle  of degree $p_a(C)-1$ and $\w_L$ be the polarization induced by $L$ on $C$.
Then $L$ is strictly balanced if and only if $\OO_C$ is $\w_L$-stable. 
\end{corollary}
\begin{proof}
Since $C$ is stable,  the exceptional locus of $C$  is empty. 
Moreover, as  we assumed $\deg(L)= p_a(C) -1$, Condition \eqref{EQ:BALANCED} is equivalent to $(\star_B)$.
This implies the claim. 
\end{proof}

\section{Good polarizations and $\w$-Stability of $\OO_C$}
\label{SEC:3}

Let $(C,\w)$ be  a polarized nodal curve. In this section we will obtain sufficient conditions for a polarization $\w$ to be good (see Theorem \ref{THM:DELTAGENERAL}). Recall that, by Corollary \ref{COR:necbounds}, any good polarization satisfies properties $(\star)_B$ of  Lemma \ref{LEM:BOUNDS}, or equivalently, is such that $\OO_C$ is $\w$-stable. 
We will show that for curves of compact type, $\w$-stability of $\OO_C$ is also sufficient in order to have $\w$ good (see Theorem \ref{THM:COMPACTTYPE}).
\\

With this aim, we will give a description of $\Delta_{\w}(E)$ as a function depending only on the residual ranks and on the contribution of the non-free part of the stalks of $E$ at nodes of $C$. We will get this description by considering paths on the dual graph of $C$, as follows.  
\\

Assume that $C$ has $\gamma$ irreducible components and $\delta$ nodes. 
Let   $C_1,\dots, C_{\gamma}$ denote the   smooth components of $C$ and $p_1,\dots, p_{\delta}$ denote the nodes of $C$.  
Let $\Gamma_C=(\cV,\cE)$ be the dual graph of $C$. It is a finite graph with   $\gamma=\#\cV$ vertices and $\delta=\#\cE$ edges.
Since $C$ is connected the same holds for $\Gamma_C$. 

\begin{notation} 
\label{NOT:GAMMADUAL}
\hfill\par\noindent
Given a path $\gamma$ in $\Gamma_C$, we will denote by $L(\gamma) \in  \mathbb{N}$ the {\bf length} of $\gamma$ i.e. the number of edges which are part of  $\gamma$. A path has length $0$ if and only if it is the trivial path. A path joining $C_i$ with $C_j$ is  said {\bf minimal} if it has minimal length among all the paths joining $C_i$ and $C_j$. As the graph  $\Gamma_C$ is connected and finite, minimal paths exist for each pair of vertices.
Two edges of $\Gamma_C$ are  said {\bf equivalent} if and only if the corresponding nodes lie on the same two components, i.e. if they connect the same vertices of $\Gamma_C$. 
\\ 

A {\bf marking} $\cM$ is a subset of $\cE$ which is a transversal for the above equivalence relation, i.e. every edge of $\Gamma_C$ is equivalent to exactly one edge in $\cM$. The subgraph $\Gamma_C^{\cM}=(\cV,\cM)$ has the same vertices of $\Gamma_C$, is connected and it is also simple (i.e. for each pair of vertices there is at most one edge). 
\\

For our construction we will need to fix arbitrarily a component of $C$. For simplicity, we will use $C_\gamma$. We define $\cP$ as any set satisfying the following properties: 
\begin{enumerate}
\item{} the elements of $\cP$ are minimal paths in $\Gamma_C^{\cM}$ connecting a vertex $C_i$ to $C_\gamma$;
\item{} for each $C_i$ there exists exactly one path in $\cP$ starting from $C_i$, which we will be denoted by $\gamma_i$;
\item{} if $\gamma_i\in\cP$ and $C_j$ is a vertex on $\gamma_i$, then $\gamma_j$ is a restriction of $\gamma_i$.
\end{enumerate}
We will call $\cP$ a {\bf set of minimal paths} of  $\Gamma_C$.  In order to simplify the notations, if $C_j\in \cV, p_k\in \cE$ we will write $p_k \subseteq \gamma_i$ if and only if $p_k$ is an edge on $\gamma_i$ and $C_j \in \gamma_i$ if and only if $C_j$ is a vertex on $\gamma_i$. We set $\cM'$ the subset of $\cM$ which consists of all the edges on some path in $\cP$.
\\

If $\gamma_i\in \cP$ and $p_j\subseteq \gamma_i$ is a node in $C_{k_1} \cap C_{k_2}$, we say that {\bf $C_{k_1}$ precedes $C_{k_2}$ with respect to $\gamma_i$} if and only if, compared to $C_{k_2}$, $C_{k_1}$ is closer to $C_{i}$ along the path $\gamma_i$. 
\end{notation}

Indeed, this does not depend on the choice of $\gamma_i\in \cP$ passing through $p_j$ as the next lemma shows.

\begin{lemma}
\label{LEM:GOODDEFPRECEDE}
Assume that $\gamma_{i_1}$ and  $\gamma_{i_2}$ are two minimal paths ending in  $C_{\gamma}$, which pass through $p_j\in \cM$ with $p_j\in C_{k_1}\cap C_{k_2}$. Then the curve $C_{k_1}$ precedes $C_{k_2}$ with respect to $\gamma_{i_1}$ if and only if the same happens with respect to $\gamma_{i_2}$.
\end{lemma}

\begin{proof}
Assume, by contradiction, that $C_{k_1}$ precedes $C_{k_2}$ with respect to $\gamma_{i_1}$ and follows $C_{k_2}$ with respect to $\gamma_{i_2}$.
For all  $l=1,2$,  we denote by $\gamma_{i_l}'$ the path obtained by $\gamma_{i_l}$ by removing all the edges before $p_j$ and by $\gamma_{i_l}''$ the path obtained by $\gamma_{i_l}'$ were we have removed also $p_j$.
Hence, $\gamma_{i_1}'$ and $\gamma_{i_2}''$ are both minimal paths (since minimality is preserved by restriction) which start from $C_{k_1}$ and end in $C_\gamma$. Similarly, $\gamma_{i_2}'$ and $\gamma_{i_1}''$ are both minimal paths connecting $C_{k_2}$ and $C_\gamma$.
As two minimal path joining the same vertices must have the same length we have
$$
\begin{cases}
L(\gamma_{i_1}')=L(\gamma_{i_2}'')=L(\gamma_{i_2}')-1\\
L(\gamma_{i_2}')=L(\gamma_{i_1}'')=L(\gamma_{i_1}')-1
\end{cases}
$$
which is clearly impossible.
\end{proof}

\begin{definition}
Let $p_j \in \cE$ corresponding to a node in $C_{k_1}\cap C_{k_2}$. If $p_j$ is equivalent to an edge which is on a path $\gamma_i\in \cP$ we say that {\bf $C_{k_1}$ precedes $C_{k_2}$} if and only if $C_{k_1}$ precedes $C_{k_2}$ with respect to $\gamma_i$. If $p_j$ is not equivalent to any edge on a path $\gamma_i\in \cP$, we choose arbitrarily one of the two possible cases ($C_{k_1}$ precedes $C_{k_2}$ or $C_{k_2}$ precedes $C_{k_1}$) making the same choice for equivalent edges.
\end{definition}

Lemma \ref{LEM:GOODDEFPRECEDE} ensures that the above definition is well posed. This gives the structure of oriented graph to $\Gamma_C$ and to its subgraph $\Gamma_C^{\cM}$.

\begin{notation}
\label{NOT:ab}
\hfill\par\noindent
Let $E$ be a depth one sheaf on $C$. Let $p_j$ be a node with $p_j \in C_{k_1}\cap C_{k_2}$. Denote by $q_{j,k_1}$ and $q_{j,k_2}$ the points of $C_{k_1}$ and $C_{k_2}$ respectively on the normalization of $C$ which are glued together in order to obtain $p_j$.
We recall that we have integers $s_j,a_{j,k_1}$ and $a_{j,k_2}$ such that
$$E_{p_j}=\OO_{p_j}^{s_j}\oplus\OO_{q_{j,k_1}}^{a_{j,k_1}}\oplus 
\OO_{q_{j,k_2}}^{a_{j,k_2}},$$
and satisfying $r_{k_l}=s_j+a_{j,k_l}$ for $l=1,2$. We set 
\begin{equation}
\label{EQ:ab}   
a_j:=a_{j,k_1} \mbox{ and }  b_j:=a_{j,k_2} \Longleftrightarrow C_{k_1} \mbox{ precedes } C_{k_2}
\end{equation}
and the opposite in the other case. In particular, we have that  $a_j+b_j=t_{p_j}$.
\end{notation}

\begin{lemma}
\label{LEM:PATH}
Let $E$ be any depth one sheaf on $C$. Then
\begin{enumerate}[(a)]
\item if $p_l$ and $p_j$ are equivalent edges,  we have $b_l-a_l=b_j-a_j$;
\item if $\gamma_i\in \cP$ then we have $\sum_{p_j\subseteq \gamma_i}(b_j-a_j)=r_{\gamma}-r_i$.
\end{enumerate}
\end{lemma}

\begin{proof}
(a) Let $E$ be a depth one sheaf. Let $p_j$ and $p_l$ be two equivalent edges.  Then $p_j,p_l\in C_{k_1}\cap C_{k_2}$. Without loss of generality we can assume that $C_{k_1}$ precedes $C_{k_2}$. Then
$$r_{k_1}=s_j+a_j=s_l+a_l\qquad r_{k_2}=s_j+b_j=s_l+b_l,$$
so $a_l-a_j=s_j-s_l=b_l-b_j$ and then $b_l-a_l=b_j-a_j$ as claimed.
\\

(b) Let $\gamma_i\in \cP$. We will prove the formula by induction on the lenght of $\gamma_i$. If $L(\gamma_i)=1$ then $\gamma_i$ is a single edge (say $p_j$) joining the vertices $C_{i}$ and $C_{\gamma}$.
Then $r_{i}=s_j+a_j, r_{\gamma}=s_j+b_j$ so $r_{\gamma}-r_i=b_j-a_j$ as claimed. Now assume that the formula is true for any minimal path of lenght at most $L$ and consider a minimal path $\gamma_i$ of lenght $L+1$. Let $p_l$ be the first edge, and denote by $C_k$ the second vertex  on the path (the first is $C_i$).
If we remove $p_l$ from the path we get, by the definition of $\cP$ the minimal path $\gamma_k$ joining $C_k$ to  $C_\gamma$ which has length $L$. So, by induction, we have
$$r_\gamma-r_k=\sum_{p_j\subseteq \gamma_k}(b_j-a_j).$$
On the other hand we have $r_{i}=s_l+a_l, r_{k}=s_l+b_l$ so $r_{k}-r_{i}=b_{l}-a_{l}$ and we have
$$r_{\gamma}-r_i=(r_k-r_i)+(r_\gamma-r_k)=(b_l-a_l)+\sum_{p_j\subseteq \gamma_k}(b_j-a_j)=\sum_{p_j\subseteq \gamma_i}(b_j-a_j)$$
as claimed.
\end{proof}

By Lemma \ref{LEM:PATH}(a) it follows that the choice of the marking $\cM$ does not influence the relation in Lemma \ref{LEM:PATH}(b). 

\begin{definition}
\label{DEF:AJ}
Assume that a marking $\cM$ and a set $\cP$ of minimal paths on $\Gamma_C$ (as in Notation \ref{NOT:GAMMADUAL}) have been chosen. Then, for any $p_j\in \cM$, we define $A_j$ to be the subcurve of $C$ with the following property: $C_i$ is a component of $A_j$ if and only if $p_j\subseteq \gamma_i$.
\end{definition}

Note that $A_j$ could be empty for same $j$: this occurs exactly when $p_j\not\in \cM'$. 
\\

Before stating the main result of this section, we will need the following technical result:

\begin{lemma}
\label{LEM:PROPAj}
Let $A_j \subseteq C$ be as in Definition \ref{DEF:AJ} and assume that $A_j$ is not empty. Then
\begin{enumerate}[(a)]
\item{} $A_j$ is a proper connected subcurve of $C$;
\item{} $A_j^c$ is connected;
\item{} $\sum_{C_i \subseteq A_j}\left(\lambda_i-\frac{\delta_i}{2}\right)=1-p_a(A_j) + (p_a(C) -1)\wrank(\OO_{A_j}) - \frac{1}{2} \delta_{A_j}=\Delta_{\w}(\OO_{A_j}) - \frac{1}{2} \delta_{A_j}$;
\item{} if $C$ is of compact type, then $\delta_{A_j}=1$.
\end{enumerate}
\end{lemma}

\begin{proof}
(a) Consider a component $C_i$ of $A_j$. Then the path $\gamma_i$ passes through $p_j$. Assume that $p_j\in C_{k_1}\cap C_{k_2}$ and that $C_{k_1}$ precedes $C_{k_2}$. Let $C_l$ be a vertex on $\gamma_i$ which is between $C_i$ and $C_{k_1}$ (included). Then $\gamma_l$ is the restriction of $\gamma_i$ and $p_j$ is an edge in $\gamma_l$. In particular, $C_l$ is a component of $A_j$. This shows that $C_i$ is connected to $C_{k_1}$ using only curves in $A_j$ so $A_j$ is connected. Properness follows as $C_{k_2}$ cannot be a component of $A_j$.
\\

(b) It is enough to show that if $C_i$ is a component not in $A_j$ then there is a path in $\Gamma_C$ from $C_i$ to $C_\gamma$ which only passes through vertices which correspond to components not in $A_j$. The path $\gamma_i$ connects $C_i$ with $C_\gamma$. Assume, by contradiction, that one of the vertex on the path $\gamma_i$, say $C_k$, is a component of $A_j$. Then, the restriction of $\gamma_i$ from $C_k$ to $C_{\gamma}$ is $\gamma_k$. Since $C_k$ is a component of $A_j$ we have that $p_j \subset \gamma_k$, so the same is true for $\gamma_i$. But this is impossible as we assumed that $C_i \not \in A_j$.
\\

(c) We denote by $C(A_j)$ and $N(A_j)$ the number of components and of nodes respectively of the curve $A_j$. We recall that $\delta_{A_j}=A_j\cdot A_j^c$ is the number of nodes of $C$ lying on $A_j$ which are not nodes of $A_j$. Then we have
\begin{multline}
\sum_{C_i \subseteq A_j}\left(\lambda_i-\frac{\delta_i}{2}\right)= \sum_{C_i \subseteq A_j} [1-g_i + w_i(p_a(C) -1)] - \frac{1}{2}\sum_{C_i \subseteq A_j}\delta_i= C(A_j) - \sum_{C_i \subseteq A_j}g_i +\\
+\wrank(\OO_{A_j}) (p_a(C) -1) -N(A_j)-\frac{1}{2}\delta_{A_j}
= 1-p_a(A_j) + (p_a(C) -1)\wrank(\OO_{A_j}) - \frac{1}{2} \delta_{A_j}
\end{multline}
as $A_j$ is connected and  $p_a(A_j)=\sum_{C_i\subseteq A_j}g_i+N(A_j)-C(A_j)+1$.  Finally,  we recall that $1-p_a(A_j) + (p_a(C) -1)\wrank(\OO_{A_j})=\Delta_{\w}(\OO_{A_j})$.
\\

(d) Since $C$ is of compact type, by (a) and (b)  it follows that $A_j$ and $A_j^c$ are both curves of compact type too.  
From Equation \eqref{EQ:paB} we have:
$$\sum_{i=1}^{\gamma}g_i =  \sum_{C_i\subseteq A_j}g_i+ \sum_{C_i\subseteq A_j^c}g_i  + \delta_{A_j}-1,$$
which implies $\delta_{A_j}= 1$.
\end{proof}

\begin{remark} 
We point out that, if $C$ is of compact type, the family of connected curves $\{ A_j\}$, defined in Definition \ref{DEF:AJ}, can be used to obtain  the conditions of $\w$-stability in \cite{Tei91} (see also Remark \ref{REM:TEIXIDOR}).
\end{remark}

We are now able to state our first result of this section: 

\begin{theorem}
\label{THM:DELTAGENERAL}
Let $(C,\w)$  be a  polarized  nodal curve. Fix a marking $\cM$ on the dual graph $\Gamma_C$ and a set  of minimal path $\cP$ as in Notations \ref{NOT:GAMMADUAL}.
Then for any  depth one sheaf $E$ we have:
$$
\Delta_{\w}(E)=\sum_{p_j \in \cM'}  \left[a_j\left(\frac{1}{2}(1 -  \delta_{A_j})+\Delta_{\w}(\OO_{A_j})\right)+b_j \left(\frac{1}{2}(1 + \delta_{A_j})-\Delta_{\w}(\OO_{A_j})\right)\right] +\frac{1}{2} \sum_{p_j \not\in\cM'} (a_j+b_j).
$$
In particular, if the conditions
\begin{equation}
\label{EQ:2star}
(\star\star)_{A_j}:\qquad \qquad \frac{1}{2}(\delta_{A_j}-1) < \Delta_{\w}(\OO_{A_j}) < \frac{1}{2}(\delta_{A_j}+1)
\end{equation}
hold for all the non-empty subcurves $A_j$ then $\w$ is a good polarization. 
\end{theorem}

\begin{proof}
We start from the expression of $\Delta_{\w}(E)$ given by Proposition \ref{PROP:DELTARESIDUAL}. Then, using Lemma \ref{LEM:PATH}(b) we have
\begin{multline*}
\Delta_{\w}(E)=\sum_{i=1}^{\gamma}r_i\left(\lambda_i-\frac{\delta_i}{2}\right)+     \frac{1}{2}\sum_{j=1}^{\delta}t_{p_j} = \sum_{i=1}^{\gamma}\left(r_{\gamma}+\sum_{p_j\subseteq \gamma_i}(a_j-b_j)\right)\left(\lambda_i-\frac{\delta_i}{2}\right)+     \frac{1}{2}\sum_{j=1}^{\delta}t_{p_j}=\\
=r_{\gamma}\sum_{i=1}^{\gamma}\left(\lambda_i-\frac{\delta_i}{2}\right)+
\sum_{i=1}^{\gamma}\sum_{p_j\subseteq \gamma_i}(a_j-b_j)\left(\lambda_i-\frac{\delta_i}{2}\right)+\frac{1}{2}\sum_{j=1}^{\delta}t_{p_j}.
\end{multline*}
By Lemma \ref{LEM:POLGEN}(a) we have that the coefficient of $r_\gamma$ in the last equality is $0$ so $\Delta_{\w}(E)$ is equal to
\begin{multline*}
\sum_{i=1}^{\gamma}\sum_{p_j\subseteq \gamma_i}(a_j-b_j)\left(\lambda_i-\frac{\delta_i}{2}\right)+\frac{1}{2}\sum_{j=1}^{\delta}(a_j+b_j) =
\sum_{j=1}^{\delta}(a_j-b_j)\sum_{\gamma_i\supseteq p_j}\left(\lambda_i-\frac{\delta_i}{2}\right)+\frac{1}{2}\sum_{j=1}^{\delta}(a_j+b_j) = \\
= \sum_{p_j\in \cM'}\left[a_j\left(\frac{1}{2}+\sum_{\gamma_i\supseteq p_j}\left(\lambda_i-\frac{\delta_i}{2}\right)\right)+b_j\left(\frac{1}{2}-\sum_{\gamma_i\supseteq p_j}\left(\lambda_i-\frac{\delta_i}{2}\right)\right)\right]
+\frac{1}{2}\sum_{p_j\not\in\cM'}(a_j+b_j)
\end{multline*}
since, if $p_j\not\in \cM'$ the sum over the path passing through $p_j$ is trivial. If $p_j\in \cM'$, the condition $\gamma_i\supseteq p_j$ is equivalent to $C_i\in A_j$ so, by Lemma \ref{LEM:PROPAj}(c) we have
\begin{equation*}
\Delta_{\w}(E)= \sum_{p_j \in \cM'}  \left[a_j\left( \frac{1}{2}+\Delta_{\w}(\OO_{A_j})- \frac{\delta_{A_j}}{2}\right)+b_j\left(\frac{1}{2}-\Delta_{\w}(\OO_{A_j})+ \frac{\delta_{A_j}}{2}\right) \right]+ \frac{1}{2}\sum_{p_j \not\in \cM'} (a_j+ b_j)
\end{equation*}
which is equal to the expression in the statement of the Theorem.
\\
Finally, if Conditions \eqref{EQ:2star} hold, we have that all the coefficients of $a_j$ and $b_j$ in the last expression of $\Delta_{\w}(E)$ are strictly positive. This proves that $\Delta_{\w}(E)\geq 0$. Moreover, if at least one among $a_j$ and $b_j$ for $j=1,\dots, \delta$ is not zero we have $\Delta_{\w}(E)>0$. Hence we have that $\Delta_{\w}(E)>0$ if and only if $E$ is locally free, i.e. $\w$ is a good polarization.
\end{proof}

With the expression given in Theorem \ref{THM:DELTAGENERAL} we are able to give a (positive) answer to Question \ref{QUE:GOODANDOCSTAB}  for  curves of compact type.

\begin{theorem}
\label{THM:COMPACTTYPE}
Let $(C,\w)$ be a polarized nodal curve of compact type. Then the collection $\{A_j\,|\, j\in \cE\}$ depends only on the choice\footnote{It is the arbitrary curve which we fix when we define  the set of minimal paths $\cP$.} of $C_{\gamma}$, for all $p_j\in \cE$ the curve $A_j$ is non-empty and we have
$$\Delta_{\w}(E)=\sum_{p_j \in \cE}  \left[a_j\left(\Delta_{\w}(\OO_{A_j})\right)+b_j \left(1-\Delta_{\w}(\OO_{A_j})\right)\right].$$
Moreover, we have that $\OO_C$ is $\w$-stable if and only if $\w$ is good.
\end{theorem}

\begin{proof}
As $C$ is of compact type we have that $\cE=\cM$ and also that $\cM=\cM'$. In fact, assume that there exists an edge $p_j \in \cM \setminus \cM'$, then $p_j = C_{k_1} \cap C_{k_2}$ and  $p_j \not\subseteq \gamma_{k_i}$, with  $\gamma_{k_i} \in \cP$. Then $\gamma_{k_1} \cup \gamma_{k_2} \cup p_j$ is the support of a cycle in $\Gamma_C$, which is impossible. The set $\cP$ is uniquely determined by the curve fixed at the beginning, i.e. on the component we have labeled $C_{\gamma}$. Then, the collection $\{A_j\,|\, j\in \cE\}$ is also uniquely determined by $C_\gamma$. Finally, since $\Gamma_C$ does not have any cycles, then   $A_j$ is non-empty for all $p_j\in \cE$.
\\

As $C$ is of compact type we have, by Lemma \ref{LEM:PROPAj}(d) that $\delta_{A_j}=1$ for all subcurve $A_j$.  With this information we can write the formula of Theorem \ref{THM:DELTAGENERAL} as follows:
$$\Delta_{\w}(E)=\sum_{p_j \in \cE}  \left[a_j\left(\Delta_{\w}(\OO_{A_j})\right)+b_j \left(1-\Delta_{\w}(\OO_{A_j})\right)\right].$$
In order to conclude the proof, by Theorem \ref{StableOC}, we only need to show that if $\OO_C$ is $\w$-stable then  $\w$ is a good polarization. Assume that $\OO_C$ is $\w$-stable. This, by Lemma \ref{LEM:BOUNDS}, is equivalent to saying $0<\Delta_{\w}(\OO_B)<\delta_B$ for all proper subcurves $B$ of $C$. In particular, for all $j$ we have $0<\Delta_{\w}(\OO_{A_j})<1$, which are the Conditions \eqref{EQ:2star} stated in Theorem \ref{THM:DELTAGENERAL}.
\end{proof}

In particular, for nodal curves of arithmetic genus $p_a(C)\leq 1$ we have a complete picture of the situation:

\begin{corollary}
\label{COR:pa01}
Let $C$  be a  nodal curve with $p_a(C) \leq 1$.
\begin{enumerate}[(a)]
    \item{} If $p_a(C) = 0$ then  any polarization on $C$  is good;
    \item{} if $p_a(C) = 1$ and $C$ is a cycle of rational curves then any polarization   is good;
    \item{} if $p_a(C) = 1$ and $C$ is of compact type then a good polarization on $C$ does not exist. 
\end{enumerate}
In particular, if $(C,\w)$ is any polarized nodal curve with $p_a(C)\leq 1$ then $\OO_C$ is $\w$-stable if and only if $\w$ is good.
\end{corollary}

\begin{proof}
(a) Let $C$ be a nodal curve with $p_a(C) = 0$. Then $C$ is of compact type and by Lemma \ref{LEM:BOUNDS}, $\OO_C$ is $\w$-stable for any polarization $\w$. By Theorem \ref{THM:COMPACTTYPE} we have that  any $\w$ is a good polarization.
\\

(b) Let $C$ be a cycle of rational curves and $\w$ a polarization.  Fix a marking $\cM$ and a set $\cP$ of minimal paths on $\Gamma_C$ and let $ \{ A_j \}$ be the subcurves defined in Definition \ref{DEF:AJ}.
By Lemma \ref{LEM:PROPAj}, for any $j$ for which $A_j$ is not empty, $A_j$ and $A_j^c$ are both proper connected subcurves of $C$,  so $p_a(A_j) = p_a(A_j^c)=0$.
From Equation \eqref{EQ:paB}, we get $\delta_{A_j}=2$.
Since $\Delta_{\w}(\OO_{A_j})= 1- p_a(A_j) = 1$, 
we have $\frac{1}{2} < \Delta_{\w}(\OO_{A_j}) < \frac{3}{2}$
which are the sufficient conditions $(\star\star)_{A_j}$ stated in Theorem \ref{THM:DELTAGENERAL}. This implies that $\w$ is good. (c) It follows by Remark \ref{REM:cycle}.
\end{proof}

Finally as an immediate consequence of Corollary 
\ref{COR:balanced} and  Theorem \ref{THM:DELTAGENERAL}, we have the following:
\begin{corollary}
\label{COR:equivbalanced}
Let $C$ be a stable nodal curve of compact type with $p_a(C) \geq 2$. Let  $L$ be a line bundle on $C$ with degree $p_a(C)-1$.
Then $\w_L$ is a good polarization if and only if $L$ is  strictly balanced. 
\end{corollary}

As we have seen in the proof of Theorem \ref{THM:COMPACTTYPE}, Conditions \eqref{EQ:2star} are really useful as they allow us to prove that, on a curve of compact type, a polarization $\w$ is good if and only if $\OO_C$ is $\w$-stable. Nevertheless, it can happen that the notion of good polarization is equivalent to the $\w$-stability of $\OO_C$ also for curves which are not of compact type (see Corollary \ref{COR:pa01} and the examples in Section \ref{SEC:EXAMPLE}). 
The reason for this is that Conditions \eqref{EQ:2star} are, in general, only sufficient. Moreover, to the authors' knowledge, there is no example of a polarized curve $(C,\w)$ with $\OO_C$ which is $\w$-stable but for which $\w$ is not good. This suggests the following conjecture:
\begin{conjecture}
\label{CONJ:OCwSTAB}
Let $(C,\w)$ be a polarized nodal curve. Then  
$\OO_C$ is $\w$-stable if and only if $\w$ is a good polarization. 
\end{conjecture}

Finally,  Theorem \ref{THM:DELTAGENERAL} allows us to  produce  an open subset of good polarizations on a stable nodal curve $C$ with $p_a(C) \geq 2$.  
For any nodal curve $C$ we will denote by 
$\mathcal W_C \subset \mathbb Q^{\gamma}$ the variety parametrizing polarizations on $C$.

\begin{lemma}
\label{LEM:GOODOPEN}
Let $C$ be a nodal curve, fix a marking $\cM$ and a set of minimal paths $\cP$. Then Conditions \eqref{EQ:2star} are open in  $\mathcal{W}_C$.
\end{lemma}
\begin{proof}
Let $ \{ A_j \}$ be the curves constructed starting from $\mathcal{P}$. Consider $\w$ and $\w'$ in $\mathcal{W}_C$ and set $\epsilon_i=w_i'-w_i$. Then we have
$$\Delta_{\w'}(\OO_{A_j}) = \Delta_{\w}(\OO_{A_j}) + (p_a(C)-1) \sum_{C_i \subseteq A_j}\epsilon_i.$$
If $\w$ satisfies Conditions $(\star\star)_{A_j}$ for all non-empty $A_j$ then one can take $\epsilon_i$ to be small enough so that $(\star\star)_{A_j}$ hold also for $\w'$.
\end{proof}

\begin{corollary}
\label{COR:familyw}
Let $C$ be a stable nodal curve with $p_a(C) \geq 2$. Then there is a non-empty open subset of $\mathcal W_C$ whose elements are good polarizations. 
\end{corollary}
\begin{proof}
Since $C$ is stable we can consider the canonical polarization $\underline{\eta}$ (see \ref{polarizationeta}). From its definition it follows that $\Delta_{\underline{\eta}}(\OO_{A_j})= \frac{1}{2}\delta_{A_j}$, so $\underline{\eta}$ satisfies condition $(\star\star)_{A_j}$.
One can then conclude by using Lemma \ref{LEM:GOODOPEN}.
\end{proof}

\section{Some examples}
\label{SEC:EXAMPLE}

In this section we propose some examples of curves (not of compact type) which we have analyzed in order to study the relation between $\w$-stability of $\OO_C$ and the fact that $\w$ is a good polarization. We underline  that we  always obtain an equivalence between these two concept. So these  are motivating examples for Conjecture \ref{CONJ:OCwSTAB}.

\begin{example}
Let $(C, \w)$ be a polarized nodal curve with two smooth irreducible components and $\delta$ nodes.
\hfill\par\rm
Let  $C_1$ and $C_2$  be the components of $C$ and  $p_1,\dots,p_{\delta}$ the nodes. If $\delta= 1$ then $C$ is of compact type and the assertion follows from Theorem \ref{THM:COMPACTTYPE}, so we will assume $\delta \geq 2$. We fix $\cM=\{p_1\}$ so that $\cP=\{\gamma_1, \gamma_2\}$, where $\gamma_1$ has  support on the edge corresponding to $p_1$ and $\gamma_2$  is trivial. We have $A_1= C_1$ and $A_2=\cdots=A_\delta=\emptyset$. Let $E$ be any depth one sheaf on $C$, by Theorem \ref{THM:DELTAGENERAL} we have:
$$\Delta_{\w}(E)=a_1\left(\frac{1}{2}(1-\delta)+\lambda_1\right)+b_1\left(\frac{1}{2}(1+\delta)-\lambda_1\right)+\frac{1}{2}\sum_{j=2}^{\delta}(a_j+b_j)$$
where, as in Lemma \ref{LEM:POLGEN}, $\lambda_i = \Delta_{\w}(\OO_{C_i})$.
Assume that $\OO_C$ is $\w$-stable, then  for $i=1,2$, we have $0 < \lambda_i < \delta$, with $\lambda_1 + \lambda_2 = \delta$. If $(\star\star)_{C_1}$ holds, i.e. if
\begin{equation}
\frac{1}{2}(\delta-1) < \lambda_1 < \frac{1}{2}(1+\delta)
\end{equation}
then by Theorem \ref{THM:DELTAGENERAL} $\w$ is good. If $(\star\star)_{C_1}$ does not hold we have either $\lambda_1\in \left(0,\frac{1}{2}(\delta-1)\right)=I_1$ or $\lambda_1\in \left(\frac{1}{2}(1+\delta),1\right)=I_2$. In the second case we have $\lambda_2\in I_1$ so up to changing the label to $C_1$ and $C_2$ we can assume $\lambda_1\in I_1$. Then we have
$$ \Delta_{\w}(E) = (b_1-a_1)\left(\frac{1}{2}(\delta-1)-\lambda_1\right)+b_1+\frac{1}{2}\sum_{j=2}^{\delta}(a_j+b_j).$$
If $b_1\geq a_1$, then we are done. Assume now that $a_1\geq b_1$. By Lemma \ref{LEM:PATH}, we have $b_j-a_j=b_1-a_1$ for all $j$ so we can write
$$\Delta_{\w}(E)=(b_1-a_1)\left(\frac{1}{2}(\delta-1)-\lambda_1\right)+\frac{1}{2}\sum_{j=1}^{\delta}b_j+\frac{1}{2}(a_1-b_1)(\delta-1)= \lambda_1(a_1-b_1)+\sum_{j=1}^{\delta}b_j.$$
Hence, also in this case we have that $\w$ is good.
\end{example}

\begin{example} 
Let $(C,\w)$ be a polarized nodal curve which is a cycle with $3$ smooth irreducible components.  
\rm
\\
Let  $C_1$, $C_2$ and $C_3$ be the components and  let $p_1$, $p_2$ and  $p_3$ the nodes. The  dual graph is a triangle with edge $p_i$ opposite to the node $C_i$.
In this case $\cE=\cM$, $\cP=\{\gamma_1,\gamma_2,\gamma_3\}$,  where $\gamma_1$ and $\gamma_2$ have  support on the edge corresponding to $p_2$ and $p_1$ respectively and $\gamma_3$ is  trivial. Then $A_1=C_2$, $A_2=C_1$ and $A_3=\emptyset$.
Let $E$ be any depth one sheaf on $C$, by Theorem \ref{THM:DELTAGENERAL} we have:
$$\Delta_{\w}(E)=a_1\left(\lambda_2-\frac{1}{2}\right)+b_1\left(\frac{3}{2}-\lambda_2\right)+a_2\left(\lambda_1-\frac{1}{2}\right)+b_2\left(\frac{3}{2}-\lambda_1\right)+\frac{1}{2}(a_3+b_3)$$
where, as above $\lambda_i = \Delta_{\w}(\OO_{C_i})$.
Assume that  $\OO_C$ is $\w$-stable, we have $0<\lambda_i<2$ with $\lambda_1 + \lambda_2 + \lambda_3 = 3$. 
If conditions $(\star\star)_{A_i}$ hold, i.e. if $\frac{1}{2}<\lambda_1,\lambda_2<\frac{3}{2}$
we can conclude.  If $(\star\star)_{A_i}$  do not hold, one can prove that by exchanging the labels to $C_1,C_2$ and $C_3$ one can assume $0<\lambda_1<\frac{1}{2}$ and $\frac{1}{2}<\lambda_2< \frac{3}{2}$.
We can write
$$\Delta_{\w}(E) = a_1 \left(\lambda_2 - \frac{1}{2}\right)  + b_1 \left( \frac{3}{2} - \lambda_2 \right)  +  \left(\frac{1}{2}- \lambda_1\right)(b_2-a_2) + b_2 + \frac{1}{2}(a_3 + b_3).$$
The cycle in the dual graph yields the following relation 
$$b_2-a_2=b_1-a_1+b_3-a_3.$$
As in the previous example using the above relation, one can prove that $\Delta_{\w}(E)\geq 0$ and equality holds if and only if $E$ is locally free, i.e. that $\w$ is good.

\end{example}

\end{document}